\documentclass[11pt]{amsart}

\usepackage{latexsym,fullpage}
\usepackage{tikz}
\usetikzlibrary{arrows}
\usepackage{pstricks-add}
\usepackage{amssymb}
\usepackage{amsfonts}
\usepackage{amsmath}
\usepackage{graphics,pstricks}
\usetikzlibrary{arrows}
\DeclareSymbolFont{legacymaths}{OT1}{cmr}{m}{n}
\SetSymbolFont{legacymaths}{bold}{OT1}{cmr}{bx}{n}
\usepackage[english]{babel}
\usepackage{times}
\usepackage[T1]{fontenc}
\usepackage{color}
\usepackage[colorlinks=true, linkcolor=blue, anchorcolor=blue, citecolor=blue, filecolor=blue, menucolor= blue, urlcolor=red]{hyperref}

\newcommand{\cL}{\mathcal{L}}
\newcommand {\PP}{\mathbb{P}}

\newtheorem{thm}{Theorem}[section]
\newtheorem{lem}[thm]{Lemma}
\newtheorem{cor}[thm]{Corollary}
\newtheorem{pro}[thm]{Proposition}

\theoremstyle{definition}

\newtheorem{defn-rem}[thm]{Definition and Remark}

\newtheorem{conj}[thm]{Conjecture}
\newtheorem{rem}[thm]{Remark}

\numberwithin{equation}{section}

\begin{document}

\title{On the strong Lefschetz question \\ for uniform powers of general linear forms in $k[x,y,z]$}

\author{Juan Migliore}
\address{Department of Mathematics,
 University of Notre Dame,
  Notre Dame,
IN 46556 \\
 USA}
 \email{migliore.1@nd.edu}

\author{Rosa Mar\'ia Mir\'o-Roig}
\address{Departament de Matemàtiques i Informàtica,
Universitat de Barcelona,
08007 Barcelona,
Spain}
\email{miro@ub.edu}

\subjclass[2010]{}

\begin{abstract}

Schenck and Seceleanu proved that if $R = k[x,y,z]$, where $k$ is an infinite field, and $I$ is an ideal generated by any collection of powers of  linear forms, then multiplication by a general linear form $L$ induces a homomorphism of maximal rank from any component of $R/I$ to the next. That is, $R/I$ has the {\em weak Lefschetz property}. Considering the more general {\em strong Lefschetz question} of when $\times L^j$ has maximal rank for $j \geq 2$, we give the first systematic study of this problem. We assume that the linear forms are general and that the powers are all the same, i.e. that $I$ is generated by {\em uniform} powers of general linear forms. We prove that for any number of such generators, $\times L^2$ always has maximal rank. We then specialize to almost complete intersections, i.e. to four generators, and we show that for $j = 3,4,5$ the behavior depends on the uniform exponent and on $j$, in a way that we make precise. In particular, there is always at most one degree where $\times L^j$ fails maximal rank. Finally, we note that experimentally all higher powers of $L$ fail maximal rank in at least two degrees.
\end{abstract}
\maketitle

\section{Introduction}

Ideals of powers of linear forms have been studied rather extensively. We can point, for example, to \cite{CHMN}, \cite{DIV}, \cite{EI}, \cite{hss}, \cite{MMN-lin forms} and \cite{SS}. We take the latter as our launching point, and we consider only ideals in $R = k[x,y,z]$, where $k$ is an infinite field.

If $R/I$ is a standard graded artinian algebra and $L$ is a general linear form, we recall that $R/I$ is said to have the {\em weak Lefschetz property (WLP)}  if the multiplication $\times L : [R/I]_{\delta-1} \rightarrow [R/I]_{\delta}$ has maximal rank for all $\delta$. The {\em strong Lefschetz property (SLP)} says that for all $j \geq 1$ the multiplication by $L^j$ has maximal rank in all degrees. We will call the {\em strong Lefschetz question} the analysis of  which $j$ and which $\delta$ provide the homomorphism $\times L^j : [R/I]_{\delta-j} \rightarrow [R/I]_\delta$ having maximal rank.

We consider ideals of the form $I = (L_1^{a_1},\dots,L_r^{a_r})$ in $R = k[x,y,z]$, where $k$ is an infinite field. A theorem of Stanley \cite{stanley} and Watanabe \cite{watanabe} shows that when $r=3$, $R/I$ has the SLP, so maximal rank always holds. Thus the question is only of interest for $r\geq 4$.

The main theorem of \cite{SS} asserts that  if $I$ is {\em any} ideal of the stated form then $R/I$ has the WLP (see also \cite{MMN-lin forms} for a different proof). This leads naturally to the question of what happens for higher powers of a general linear form.
For $\times L^2$ it was shown in \cite{MMN-lin forms} that for $r = 4$, if the linear forms are chosen generally then $\times L^2 : [R/I]_j \rightarrow [R/I]_{j+2}$ has maximal rank for all $j$.
On the other hand, it was shown in \cite{CHMN} and \cite{DIV} that if the linear forms are not required to be general then $\times L^2$ does not necessarily have maximal rank, and indeed the question of maximal rank is a quite subtle one depending on the geometry of the set of points dual to the linear forms. Thus we focus on general linear forms.

So what, exactly, should we expect for $\times L^j$ for $L$ a general linear form and $j \geq 2$?
In this paper we want to begin the study of the multiplication by higher powers, $L^j$, of the general linear form by assuming that the exponents of the linear forms generating our ideal are all the same, i.e. that we have {\em uniform} powers. Our first main result, Theorem \ref{L^2}, is that for arbitrary $r$, $\times L^2 : [R/I]_{\delta-2} \rightarrow [R/I]_\delta$ has maximal rank for all $\delta$. We conjecture that in fact the result also holds for mixed powers.

For $j \geq 3$ we already get interesting behavior for $r=4$, i.e. by assuming that the ideal is an almost complete intersection of uniform powers of general linear forms: $I = (L_1^k,\dots,L_4^k)$. We want to see if $\times L^j$ always has maximal rank, and if not, to see how often we can expect this phenomenon to occur. We find that  it is rarely the case that $\times L^j$ has maximal rank in all degrees, in fact, but it occasionally does. In this paper we classify those values of $j$ and $k$ for which it does have maximal rank in all degrees, and those values of $j$ and $k$ for which it fails maximal rank in only one degree.

More precisely, in Theorem \ref{L3}, Theorem \ref{L4} and Theorem \ref{L5} we show that $\times L^3$ and $\times L^4$ sometimes have maximal rank in all degrees, depending on the congruence class of $k$ modulo 3, and that $\times L^5$ never has maximal rank in all degrees. However, we also show that for $j = 3,4,5$, whenever this multiplication fails maximal rank, it does so only in one spot. We note in Remark \ref{geq 6} that for higher powers of $L$, the multiplication fails maximal rank in more than one spot.
These results show more clearly that Anick's theorem does not extend from general forms to powers of general linear forms, although this was already known. (Indeed, if $I=(x^3,y^3,z^3,L_1^3)\subset R$, where $L_1$ is a general linear form, then for a general linear form $L$,  $\times L^3$ fails to have maximal rank (see Proposition \ref{socledegree}), while Anick's result shows that if instead we take general forms of degree 3 then maximal rank does hold.)


\section{Preliminaries}

Throughout this paper we consider the homogeneous polynomial ring $R = k[x,y,z]$, where $k$ is an infinite field. In this section we recall the main tools that we will use in the rest of the paper.

For any artinian ideal $I \subset R$ and a general linear form $L \in R$, the exact sequence
\[
\cdots \rightarrow [R/I]_{m-j} \stackrel{\times L^j}{\longrightarrow} [R/I]_m \rightarrow [R/(I,L^j)]_m \rightarrow 0
\]
gives, in particular, that the multiplication by $L^j$  will fail to have maximal rank exactly when
\begin{equation}
  \label{eq:max-rank-I}
\dim_k [R/(I,L^j)]_m \neq \max\{ \dim_k [R/I]_m - \dim_k [R/I]_{m-j} , 0 \};
\end{equation}
in that case, we will say that  $R/I$  fails maximal rank in degree $m$.

We will deeply need the following result of Emsalem and Iarrobino, which gives a duality between powers of linear forms and ideals of fat points in $\mathbb P^{n-1}$.  We only quote Theorem I in \cite{EI} in the form that we need.

\begin{thm}[\cite{EI}]   \label{thm:inverse-system}
Let $\langle L_1^{a_1} ,\dots,L_n^{a_n} \rangle \subset R$ be an ideal generated by powers of $n$ general linear forms.  Let $\wp_1, \dots, \wp_n$ be the ideals of $n$ general points in $\mathbb P^{2}$.  (Each point is actually obtained explicitly from the corresponding linear form by duality.)  Choose positive integers $a_1,\dots,a_n$.  Then for any integer $j \geq \max\{ a_i  \}$,
\[
\dim_k \left [R/ \langle L_1^{a_1}, \dots, L_n^{a_n} \rangle  \right ]_j =
\dim_k \left [ \wp_1^{j-a_1 +1} \cap \dots \cap \wp_n^{j-a_n+1} \right ]_j .
\]
\end{thm}

From now on, we will denote by
$${\mathcal L}_{2}(j; b_1, b_2,\cdots ,b_n)$$
the linear system  $[ \wp_1^{b_1} \cap \dots \cap \wp_n^{b_n} ]_j\subset [R]_j $. Note that we view it as a vector space, not a projective space, when we compute dimensions.
If necessary, in order to simplify notation, we use superscripts to indicate repeated entries. For example,
$\cL_2(j; 5^2, 2^3) = \cL_2 (j; 5, 5, 2, 2, 2)$.

Notice that, for every linear system $\cL_2 (j;ba_1,\ldots,b_n)$, one has
\[
\dim_k  \cL_2 (j; b_1,\ldots,b_n) \ge \max \left \{0, \binom{j+2}{2} - \sum_{i=1}^n \binom{b_i +1}{2} \right \},
\]
where the right-hand side is called the {\em expected dimension} of the linear system. If the inequality is strict, then the linear system $\cL_2 (j; b_1,\ldots,b_n)$ is called {\em special}. It is a difficult problem to classify the special linear systems.

Using Cremona transformations, one can relate two different linear systems (see \cite{Nagata}, \cite{LU}, or \cite{Dumnicky}, Theorem 3), which we state only in the form we will need even though the cited results are more general.

\begin{lem}
  \label{lem:Cremona}
Let $n >   2$ and let $j, b_1,\ldots,b_n$ be non-negative integers, with $b_1 \geq \dots \geq b_n$.  Set $m =  j - (b_1 + b_2 + b_{3})$. If $b_i + m \ge 0$ for all $i = 1,2,3$, then
\[
\dim_k \cL_2 (j; b_1,\ldots,b_n) = \dim_k \cL_2 (j + m; b_1 +m,b_2+m,b_{3} +m , b_{4},\ldots,b_n).
\]
\end{lem}

The analogous linear systems have also been studied for points in $\mathbb P^r$. Following \cite{DL}, the linear system $\cL_r (j; b_1,\ldots,b_n)$ is said to be in {\em standard form} if \[
(r-1)  j \ge b_1 + \dots  + b_{r+1} \quad \text{and} \quad b_1 \ge \cdots \ge b_n \ge 0.
\]
In particular, for $r=2$,  they show that every linear system in standard form is non-special.  (This is no longer true if $r \ge 3$. For example,  $\cL_3 (6; 3^9)$ is in standard form and special.)

Notice again that we always use the vector space dimension of the linear system rather than the dimension of its projectivization. Furthermore, we always use the convention that a binomial coefficient $\binom{a}{r}$ is zero if $a < r$.

\begin{rem} \label{bezout}
B\'ezout's theorem also provides a useful simplification. Again, we only state the result we need in this paper. Assume the points $P_1,\dots,P_n$ are general.  If $2j < b_1 + \dots +b_5$ then
\[
\dim \cL_2(j;b_1,\dots,b_n) = \dim \cL_2 (j-2; b_1-1,\dots,b_5-1,b_6,\dots, b_n).
\]
If $j < b_1 + b_2$ then
\[
\dim \cL_2(j;b_1,\dots,b_n) = \dim \cL_2 (j-1; b_1-1,b_2-1,b_3,\dots, b_n).
\]
\end{rem}

\begin{lem} \label{res of ci}
Let $P_1,\dots,P_4$ be general points in $\mathbb P^2$ with homogeneous ideals $\wp_1,\dots, \wp_4$ respectively, and let $X = \{ P_1 ,\dots, P_4 \}$. Let $m \geq 1$ be an  integer. Then $I_X^m$ is a saturated ideal, and  the minimal free resolution of $I_X^m$ has the form
\[
0 \rightarrow R(-2m-2)^m   \rightarrow R(-2m)^{m+1} \rightarrow I_X^m \rightarrow 0.
\]
In particular, $I_X^m = \wp_1^m \cap \dots \cap \wp_4^m = I_X^{(m)}$.
\end{lem}

\begin{proof}
This is well known, since $X$ is the reduced complete intersection of two conics. See for instance \cite{HU}, Theorem 2.8 or \cite{powers}, Corollary 2.10.
\end{proof}


\section{Preparation} \label{preparation section}

From now on we will consider quotients of the form $R/I$, where $R = k[x,y,z]$, $I = (L_1^k,\dots,L_r^k)$ and $L_1,\dots,L_r$ are general linear forms. Specifically, we are interested in whether
\[
\times L^j : [R/I]_{\delta-j} \rightarrow [R/I]_\delta
\]
has maximal rank for all $\delta$, for $j = 2, 3,4$ and 5, where $L$ is a general linear form. Since the case $k=2$ is trivial, we assume $k \geq 3$. In section \ref{section L2} we work with arbitrary $r$, but in section \ref{section higher j} we restrict to $r=4$. In this section we  give technical preparatory results that will be central to our proofs in section \ref{section higher j}. Thus from now on in this section we assume $r=4$ (and we return to arbitrary $r$ in section \ref{section L2}). However, the general approach used in section \ref{section L2} will also be reflected in our preparation in this section.

We first compute the socle degree (i.e. the last non-zero component) of $R/I$. Since $L_1,L_2,L_3$ are general, without loss of generality we can assume that $L_1 = x, L_2 = y, L_3 = z$. Then by a well-known result of Stanley \cite{stanley} and Watanabe \cite{watanabe}, $\times L_4^k$ has maximal rank in all degrees. The socle degree of $R/I$ is the last degree where $\times L_4^k$ is not surjective. Since $R/(L_1^k,L_2^k,L_3^k)$ has socle degree $3k-3$, one checks that the socle degree of $R/I$ is $2k-2$. (This also follows from \cite{MM}, Lemma 2.5.)

More precisely, we make the following Hilbert function calculation, also using the fact that the Hilbert function of $R/(x^k,y^k,z^k)$ is symmetric, and that of $R/I$ ends in degree $2k-2$.
{\small
\[
\begin{array}{c|cccccccccccccccc}
\hbox{degree} & 0 & 1 & 2 & \dots & k-2 &  k-1 & k & k+1 & \dots & 2k-4 & 2k-3 & 2k-2 \\ \hline
R/(x^k,y^k,z^k) & 1 & 3 & 6 & \dots & \binom{k}{2} & \binom{k+1}{2} & \binom{k+2}{2}-3 & \binom{k+3}{2}-9 & \dots & \binom{k+3}{2}-9 & \binom{k+2}{2}-3 & \binom{k+1}{2} \\
R/I & 1 & 3 & 6 & \dots & \binom{k}{2} & \binom{k+1}{2} & \binom{k+2}{2}-4 & \binom{k+3}{2} - 12 & \dots & 5k-9 & 3k-3 & k
\end{array}
\] }

In \cite{MMN-monomials} Proposition 2.1, it was observed that for any standard graded algebra $R/I$, if $\times L : [R/I]_{\delta-1} \rightarrow [R/I]_\delta$ is surjective then so is $\times L : [R/I]_{\delta+i} \rightarrow [R/I]_{\delta+i+1}$ for all $i \geq 0$. The same clearly holds for $\times L^j$ (after adjusting the indices).
Furthermore, {\em if $R/I$ is level} and $\times L : [R/I]_{\delta-1} \rightarrow [R/I]_\delta$ is injective then so is $\times L : [R/I]_{\delta-i} \rightarrow [R/I]_{\delta-i+1}$ for all $i \geq 2$. In our present situation, we conjecture that $R/I$ is always level:

\begin{conj}
{\em If $R = k[x,y,z]$ and $I = (L_1^k,\dots,L_4^k)$ with $L_1,\dots,L_4$ general, then $R/I$ is level with Cohen-Macaulay type $k$}.
\end{conj}

\noindent However, since we are interested in multiplication by higher powers of $L$, it turns out that we do not need $R/I$ to be level, as we now show.

\begin{lem} \label{tool}
Let $M$ be a graded module generated in the first $m$ degrees, say $b, b+1,\dots, b+m-1$, for some $m \geq 1$. Let $L$ be a general linear form. If $j \geq m$ and $\times L^j : [M]_b \rightarrow [M]_{b+j}$ is surjective, then $\times L^j : [M]_{b+i} \rightarrow [M]_{b+i+j}$ is also surjective, for all $i \geq 0$.
\end{lem}

\begin{proof}
The module $M/(L^jM)$ is generated in degree $\leq b+m-1$ and is zero in degree $b+j \geq b+m$, hence is zero thereafter.
\end{proof}

\begin{lem} \label{socle}
Let $I = (L_1^k,\dots,L_4^k)$, where $L_1,\dots,L_4$ are general linear forms. Then the socle of $R/I$ occurs in degree $2k-2$ and possibly in degree $2k-3$.
\end{lem}

\begin{proof}
Since the socle degree of $R/I$ is $2k-2$, we just have to show that $R/I$ has no socle in degree $\leq 2k-4$. The ideal $(L_1^k,L_2^k,L_3^k)$ is a complete intersection, linking the  almost complete intersection $I$ to a Gorenstein ideal $J$.  Using the formula for the Hilbert function of artinian algebras under liaison (see \cite{DGO}) and the above Hilbert function calculation, we see that $R/J$ has socle degree $(3k-3) - k = 2k-3$ and Hilbert function
\[
\left (1,3,6, \dots , \binom{k-2}{2}, \binom{k-1}{2}, \binom{k}{2}, \binom{k}{2}, \binom{k-1}{2}, \binom{k-2}{2}, \dots, 6,3,1 \right ).
\]
Let us consider the minimal free resolutions. That of $I$ has the form
\[
\begin{array}{cccccccccccccccccc}
0 & \rightarrow &
\begin{array}{c}
R(-2k-1)^k \\
\oplus \\
R(-2k)^a \\
\oplus \\
F
\end{array}
& \rightarrow & G & \rightarrow & R(-k)^4 & \rightarrow & I & \rightarrow & 0
\end{array}
\]
where
\[
\begin{array}{l}
a  \geq  0; \\ \\
\displaystyle F  =   \bigoplus_{k+2 \leq i \leq 2k-1} R(-i)^\bullet \\ \\
\displaystyle G =  \bigoplus_{k+1 \leq i \leq 2k} R(-i)^\bullet
\end{array}
\]
(we do not care what the exponents of the components of $F$ are because we will show $F = 0$; nor do we care what the exponents of the components of $G$ are).

Linking $I$ by the complete intersection, the standard mapping cone construction (splitting three copies of $R(-k)$) gives a free resolution for $J$:
\[
\begin{array}{ccccccccccccccccc}
0 & \rightarrow R(-2k) & \rightarrow &
G^\vee (-3k) & \rightarrow &
\begin{array}{c}
R(1-k)^k \\
\oplus \\
R(-k)^a \\
\oplus \\
F^\vee (-3k)
\end{array}
& \rightarrow & J & \rightarrow & 0
\end{array}
\]
where
\[
G^\vee (-3k) = R(1-2k)^\bullet \oplus R(2 - 2k)^\bullet \oplus \dots \oplus R(-k)^\bullet
\]
and
\[
F^\vee(-3k) = R(2-2k)^\bullet \oplus \dots \oplus R(-k-1)^\bullet.
\]
Now, any summand of $G^\vee (-3k)$ of the form $R(-i)$ for $i  \geq k+2$ must correspond, by the duality of the resolution, to a minimal generator of degree $2k-i \leq k-2$, which is forbidden by the Hilbert function. But any minimal generator of $J$ must be represented in this way, so $J$ only has generators of degrees $k-1$ and $k$, and $F = 0$ as desired. But returning to the minimal free resolution of $I$, this means that the socle of $R/I$ is as claimed.
\end{proof}

The following consequence allows us to confirm the maximal rank property for $\times L^j$ by checking only two degrees (which sometimes coincide).

\begin{cor} \label{two degrees}
Let $I = (L_1^k,\dots,L_4^k)$ as above. Let $j \geq 2$. Let
\[
\begin{array}{rcl}
a &  : = & \max \{ \delta \ | \ h_{R/I}(\delta-j) \leq h_{R/I}(\delta) \} \\
b & := & \min \{ \delta \ | \ h_{R/I}(\delta-j) \geq h_{R/I}(\delta) \}.
\end{array}
\]
If $\times L^j : [R/I]_{a-j} \rightarrow [R/I]_{a}$ is injective and $\times L^j : [R/I]_{b-j} \rightarrow [R/I]_{b}$ is surjective then $\times L^j$ has maximal rank in all degrees.
\end{cor}

\begin{proof}
The fact that $\times L^j$ is surjective in all degrees $\geq b$ was noted above and is standard. We have to show the analogous result for injectivity of $\times L^j$ for all degrees $\leq a$.

Consider  the canonical module, $M$, of $R/I$. Since $R/I$ is artinian, $M$ is isomorphic to a shift of the $k$-dual of $R/I$. The injectivity of $\times L^j : [R/I]_a \rightarrow [R/I]_{a+j}$ is equivalent to the surjectivity of the dual homomorphism on $M$, say from $[M]_{a'}$ to $[M]_{a'+j}$. By abuse of notation we continue to write this as $\times L^j$.

By Lemma \ref{socle}, $M$ is generated in the first two degrees, at most. If $a'$ is not the initial degree of $M$, let $N$ be the truncation of $M$ in degree $a'$, i.e. $N = \bigoplus_{i \geq a'} [M]_i$. $N$ is generated in the first degree, unless $N=M$ in which case it may be generated in the first two degrees. Either way, Lemma \ref{tool} gives that $\times L^j$ is surjective in all degrees $\geq a'$. Then by  duality,  $\times L^j$ is injective in all degrees $\leq a$.
\end{proof}

\begin{rem} \label{find a b}
Of course it is important to determine the values of $a$ and $b$ in order to be able to apply Corollary \ref{two degrees}. Our method will be to take advantage of the fact that four general points in $\mathbb P^2$ are a complete intersection, and use Lemma \ref{res of ci}. We  note here  that we will implicitly use the fact that the Hilbert function is unimodal (a fact that is true not just for four powers of linear forms but in fact for any ideal generated by powers of linear forms in $k[x,y,z]$), which is an immediate consequence of the fact that the algebra has the weak Lefschetz property \cite{SS}, so the unimodality follows from \cite{HMNW} Remark 3.3.
\end{rem}

The following is central to determining the values of $a$ and $b$ in Corollary \ref{two degrees}.
For any $\delta$ we have the exact sequence
\begin{equation} \label{std exact seq}
[R/(L_1^k,\dots,L_4^k)]_{\delta-j} \stackrel{\times L^j}{\longrightarrow} [R/(L_1^k,\dots,L_4^k)]_{\delta} \rightarrow
[R/(L_1^k,\dots,L_4^k,L^j)]_{\delta} \rightarrow 0.
\end{equation}
Let $\wp_i$ be the point dual to $L_i$ and let $\wp_1 \cap \dots \cap \wp_4 = I_X$. We will define the following functions of $\delta, j$ and~$k$.
\[
C_1 = \dim [R/(L_1^k, \dots, L_4^k)]_{\delta-j} \ \ \hbox{ and } \ \ C_2 = \dim [R/(L_1^k, \dots, L_4^k)]_{\delta} .
\]
We would like to apply Theorem \ref{thm:inverse-system}. It is certainly no loss of generality to assume that $\delta \geq k$ since in smaller degrees $R/I$ coincides with the polynomial ring, where maximal rank  holds. Thus we have
\[
\begin{array}{cccccccc}
C_2 & = & \dim [R/(L_1^k,\dots,L_4^k)]_{\delta} & = & \dim [\wp_1^{\delta-k+1} \cap \dots \cap \wp_4^{\delta-k+1}]_{\delta} & = & \dim [I_X^{\delta-k+1}]_{\delta} .
\end{array}
\]
We also have
\[
\begin{array}{llllllll}
C_1 & = & \dim [R/(L_1^k,\dots,L_4^k)]_{\delta-j} \\ \\
& = & \left \{
\begin{array}{ll}
\dim [R]_{\delta-j},  & \hbox{if } \delta \leq j+k-1;\\ \\
 \dim [\wp_1^{\delta-j-k+1} \cap \dots \cap \wp_4^{\delta-j-k+1}]_{\delta-j}  =  [I_X^{\delta-j-k+1}]_{\delta-j}, & \hbox{if } \delta \geq j+k-1  . \\
 \end{array}
 \right.
 \end{array}
 \]

Since the Hilbert function of $R/I$ is unimodal (Remark \ref{find a b}), we simply need to set $C_1 - C_2$ equal to zero and find the nearest integer values for $\delta$, as we make precise now. We make use of Lemma \ref{res of ci}.

\medskip

\underline{Case 1}: \  First we assume that $\delta \geq j+k-1$. Notice that we adopt the convention that $I_X^0 = R$. We have the resolutions
\[
0 \rightarrow R(-2\delta+2j+2k -4)^{\delta-j-k+1} \rightarrow R(-2\delta +2j+ 2k-2)^{\delta-j-k+2} \rightarrow I_X^{\delta -j -k+1} \rightarrow 0
\]
and
\[
0 \rightarrow R(-2\delta+2k -4)^{\delta-k+1} \rightarrow R(-2\delta + 2k-2)^{\delta-k+2} \rightarrow I_X^{\delta -k+1} \rightarrow 0.
\]
Using this, we have
\[
\begin{array}{rcl}
C_1-C_2 & = & \displaystyle
(\delta -j-k+2) \binom{-\delta +j+2k}{2} - (\delta -j-k+1) \binom{-\delta+j+2k-2}{2} \\ \\
&& \displaystyle - (\delta -k+2) \binom{-\delta +2k}{2} + (\delta -k+1) \binom{-\delta + 2k-2}{2}
\end{array}
\]
from which an elementary but tedious calculation gives
\begin{equation} \label{critical values}
C_1 - C_2 = 3 j \delta  - 4 kj - 3 \binom{j-1}{2} + 3.
\end{equation}
In particular, we have the following values.

\begin{equation} \label{critical values 2}
\begin{array}{c|ccccccccc}
j & C_1 - C_2 \\ \hline
2 & 6 \delta - 8 k + 3 \\
3 & 9 \delta - 12 k \\
4 & 12 \delta - 16 k -6 \\
5 & 15 \delta - 20 k -15
\end{array}
\end{equation}

\begin{rem} \label{values of a,b}
Notice that when $\delta \geq j+k-1$ we have
\[
\begin{array}{rcl}
a & = & \max \{ \delta \in \mathbb Z \ | \ C_1 - C_2 \leq 0 \} \\
b & = & \min \{ \delta \in \mathbb Z \ | \ C_1 - C_2 \geq 0 \}.
\end{array}
\]
More precisely, using (\ref{critical values}), an easy calculation gives that for $\delta \geq j+k-1$ we get

\begin{itemize}

\item If $j$ is odd then
\[
a =
\left \{
\begin{array}{ll}
4k_0 + \frac{j-1}{2} -1, & \hbox{if $k = 3k_0$}; \\
4k_0 + \frac{j-1}{2}  ,  & \hbox{if $k = 3k_0+1$}; \\
4k_0 + \frac{j-1}{2} +1 , & \hbox{if $k = 3k_0+2$}.
\end{array}
\right.
\]
\[
b =
\left \{
\begin{array}{ll}
4k_0 + \frac{j-1}{2} -1, & \hbox{if $k = 3k_0$}; \\
4k_0 + \frac{j-1}{2} +1 ,  & \hbox{if $k = 3k_0+1$}; \\
4k_0 + \frac{j-1}{2} +2 , & \hbox{if $k = 3k_0+2$}.
\end{array}
\right.
\]

\item If $j$ is even then
\[
a =
\left \{
\begin{array}{ll}
4k_0 + \frac{j}{2} -2 ,& \hbox{if $k = 3k_0$}; \\
4k_0 + \frac{j}{2} -1 ,  & \hbox{if $k = 3k_0+1$}; \\
4k_0 + \frac{j}{2} +1  ,& \hbox{if $k = 3k_0+2$}.
\end{array}
\right.
\]
\[
b =
\left \{
\begin{array}{ll}
4k_0 + \frac{j}{2} -1 ,& \hbox{if $k = 3k_0$}; \\
4k_0 + \frac{j}{2}  ,  & \hbox{if $k = 3k_0+1$}; \\
4k_0 + \frac{j}{2} +2 , & \hbox{if $k = 3k_0+2$}.
\end{array}
\right.
\]

\end{itemize}
\end{rem}

\medskip

\underline{Case 2}: \  Now we assume that $k \leq \delta \leq j+k-2$. Then
\[
\begin{array}{rcl}
C_1 - C_2 & = & \displaystyle \binom{\delta -j+2}{2} - (\delta-k+2) \binom{- \delta +2k}{2} + (\delta-k+1)\binom{-\delta+2k-2}{2} \\ \\
& = & \displaystyle \binom{\delta-j+2}{2} - \binom{- \delta +2k}{2} + (2 \delta -4k+3)(\delta-k+1).
\end{array}
\]

\begin{rem} \label{bad delta}
In proving our main results in the next section, an important issue is  that we have two formulas for the value of $C_1 - C_2$, depending on the relation between $\delta, j$ and $k$. The value of $C_2$ is not at issue, but the value of $C_1$ is. We would like to use our formulas from Remark \ref{values of a,b} to make our calculations in the proofs given in the next section. However, we  sometimes need to use  values of $\delta $ as low as $a-1$, and we need to understand which values of $\delta$, $j$ and $k$ force us to use Case 2 above instead of Case 1.


\end{rem}


\section{Multiplication by $L^2$} \label{section L2}

For ideals generated by powers of linear forms in $k[x,y,z]$, the following two results are known.

\begin{thm} [\cite{SS}, main theorem]
An artinian quotient of $k[x,y,z]$ by powers of (arbitrary) linear forms has WLP.
\end{thm}

In the following sections we will see that multiplication by $L^j$ for $j \geq 3$ does not necessarily have maximal rank, even when $I$ is an almost complete intersection. This leaves the question of $\times L^2$. In \cite{DIV} and \cite{CHMN} it is shown that there exist ideals generated by powers of linear forms for which $\times L^2$ does not have maximal rank in all degrees, so the remaining question is what happens for powers of general linear forms. When $I$ is an almost complete intersection we have the following result:

\begin{thm} [\cite{MMN-lin forms}, Proposition 4.7]
Let $L_1,\dots,L_4, L$ be five general linear forms of $R = k[x,y,z]$. Let $I$ be the ideal $(L_1^{a_1},\dots,L_4^{a_4})$. Let $A = R/I$. Then, for each integer $j$, the multiplication map $\times L^2 : [A]_{j-2} \rightarrow [A]_j$ has maximal rank.
\end{thm}

Improving on this result, our next goal will be to prove that if $I$ is generated by  {\em any} number, $r$, of uniform powers of general linear forms then $R/I$ has the property that $\times L^2$ has maximal rank in all degrees. Since the case $r\le 4$ is already known, we will assume that $r\ge 5$. Recall that  {\em any} ideal generated by uniform powers of general linear forms has WLP. In particular, its Hilbert function is unimodal and we will now determine its peak(s).

\begin{lem} \label{key_lem}
Let $L_1,\dots,L_r\in k[x,y,z]$ be $r \geq 5$ general linear forms. Let $I$ be the ideal $(L_1^{k},\dots,L_r^{k})$. Write $k=(r-1)k_0+e$ with $0\le e\le r-2$. It holds:
\begin{itemize}
\item[(i)] If $2\le k \le r-2$ then $R/I$ has exactly one peak at $k-1$.
\item[(ii)] If $k_0\ge 1$ and $1\le e\le r-2$ then $R/I$ has exactly one peak at $rk_0+e-1.$
\item[(iii)] If $k_0\ge 1$ and $ e=0$ then $R/I$ has exactly two peaks at $rk_0-2$ and $rk_0-1.$
\end{itemize}
\end{lem}

\begin{proof} (i) For $2\le k \le r-2$, we have $$\dim [R/I]_{k-2}=\dim R_{k-2}={k\choose 2},$$
$$\dim [R/I]_{k-1}=\dim R_{k-1}={k+1\choose 2} \text{ and }$$
$$\dim [R/I]_{k}={k+2\choose 2}-r.$$
Hence, $R/I$ has a peak at $k-1$.

\vskip 2mm
(ii) Let us first assume that $k_0\ge 2$. We call $A_i=\dim[R/I]_{rk_0+e-1+i}$ with $i=-1,0,1$. We have to prove that $A_0-A_i>0$ for $i=-1,1$. Let us compute $A_i$ for $i=-1,0,1$. Since $k_0\ge 2$ we can apply Theorem \ref{thm:inverse-system} and we get
$$
A_i=\dim[R/I]_{rk_0+e-1+i}=\dim_k \cL_2 (rk_0+e-1+i;(k_0+i)^r)={rk_0+e+1+i\choose 2}-r{k_0+i+1\choose 2}
$$
where the last equality follows from the fact that the linear system $\cL_2 (rk_0+e-1+i;k_0+i^r)$ is in standard form and, hence, it is non-special.
Now, we easily check that $A_0-A_{-1}=e>0$ and $ A_0-A_1=r-1-e>0.$ Therefore, $R/I$ has a peak at $rk_0+e-1$.

For $k_0=1$ we have  $$\dim [R/I]_{k-1}=\dim R_{k-1}={k+1\choose 2},$$
$$
\dim [R/I]_{k}={k+2\choose 2}-r \text{ and }
$$
$$
\dim [R/I]_{k+1}= \dim \cL_2 (k+1;2^r)={k+3\choose 2}-3r.
$$
Since $\dim [R/I]_{k}-\dim [R/I]_{k-1}=k+1-r>0$ and $\dim [R/I]_{k}-\dim [R/I]_{k+1}=2r-(k+2)>0$, $R/I$ has a peak at $k$.

\vskip 2mm
(iii) First we assume that $k_0\ge 2$. Call $B_i=\dim[R/I]_{rk_0-2+i}$ with $i=-1,0,1,2$. We have
$$
B_i=\dim[R/I]_{rk_0-2+i}=\dim_k \cL_2 (rk_0-2+i; (k_0-1+i)^r)={rk_0+i\choose 2}-r{k_0+i\choose 2}.
$$
Notice that if $k_0=2$, we have $B_{-1}={2r-1\choose 2}$. In all cases, we get $B_0=B_1$, $B_0-B_2=r-1$ and $B_0-B_{-1}=r-1$ and we conclude that $R/I$ has exactly two peaks, at $rk_0-2$ and $rk_0-1$.

Finally for $k_0=1$, we have $B_{-1}=\dim[R]_{r-3}={r-1\choose 2}$,  $B_{0}=\dim[R]_{r-2}={r\choose 2}$ and
$$
B_i=\dim[R/I]_{r-2+i}=\dim_k \cL_2 (r-2+i;i^r)={r+i\choose 2}-r{i+1 \choose 2}
$$
for $i=1,2$. Again $B_0=B_1$, $B_0-B_2=r-1$, $B_0-B_{-1}=r-1$ and  $R/I$ has exactly two peaks, at $r-2$ and $r-1$.
\end{proof}

\begin{thm} \label{L^2}
 Let $L_1,\dots,L_r,L\in k[x,y,z]$ be $r+1$ general linear forms. Let $I$ be the ideal $(L_1^{k},\dots,L_r^{k})$. Then, for each integer $j$, the multiplication map
 $$\times L^2:[R/I]_{j-2}\longrightarrow [R/I]_{j}$$
 has maximal rank.
\end{thm}

\begin{proof} We write $k=(r-1)k_0+e$ with $0\le e\le r-2$ and we distinguish two cases:

\medskip

\noindent \underline{Case 1:} $k_0\ge 1$. We distinguish 3 subcases:

\medskip

1.1.- Assume $e=0$. In this case the result follows from Lemma \ref{key_lem} and the fact that, for any integer $j$, the multiplication map $\times L:[R/I]_{j-1}\longrightarrow [R/J]_j$ has maximal rank.

\medskip

1.2.- Assume $1\le e\le \frac{r-1}{2}$. By Lemma \ref{key_lem}, $R/I$ has exactly one peak, at $rk_0+e-1$ and, moreover, $A_{-1}=\dim[R/I]_{rk_0+e-2} \geq A_{1}=\dim[R/I]_{rk_0+e}$. So, we only need to check that $[R/(I,L^2)]_{rk_0+e}=0$ since this will imply the surjectivity of $\times L^2:[R/I]_{rk_0+e-2}\longrightarrow [R/I]_{rk_0+e}$. We have

\[
\begin{array}{rcll}
\dim \hbox{ coker} (\times L^2)_{rk_0+e} & = & \dim [R/(L_1^{(r-1)k_0+e}, \dots, L_r^{(r-1)k_0+e}, L^2)]_{rk_0+e} & (\hbox{by  (\ref{std exact seq}))}\\
& = & \dim [\wp_1^{k_0+1} \cap \dots \cap \wp_r^{k_0+1} \cap \wp^{rk_0+e-1}]_{rk_0+e} & (\hbox{by Theorem \ref{thm:inverse-system})} \\
& = & \dim \cL_2 (rk_0+e; (k_0+1)^r,rk_0+e-1) \\
& = & \dim \cL_2 (rk_0+e-r;k_0^r,rk_0+e-1-r) & (\hbox{by Remark \ref{bezout})} \\
& = & \cdots & (\hbox{by Remark \ref{bezout})}\\
& = & \dim \cL_2 (e;1^r,e-1) & (\hbox{by Remark \ref{bezout})} \\
& = & 0 & (\hbox{because $2e+1\le r$ )}
\end{array}
\]

1.3.- Assume $\frac{r-1}{2}<e\le r-2$. By Lemma \ref{key_lem}, $R/I$ has exactly one peak at $rk_0+e-1$ and, moreover, $\dim[R/I]_{rk_0+e}-\dim[R/I]_{rk_0+e-2}=A_{1}-A_{-1}=(A_{1}-A_{0})-(A_{-1}-A_{0})=2e-r+1 > 0$. Hence we have to show that $\times L^2$ is injective, with cokernel of dimension $2e-r+1$.  Let us compute $\dim[R/(I,L^2)]_{rk_0+e}$. As above we have
\[
\begin{array}{rcll}
\dim \hbox{ coker} (\times L^2)_{rk_0+e} & = & \dim [R/(L_1^{(r-1)k_0+e}, \dots, L_r^{(r-1)k_0+e}, L^2)]_{rk_0+e} & (\hbox{by  (\ref{std exact seq}))}\\
& = & \dim [\wp_1^{k_0+1} \cap \dots \cap \wp_r^{k_0+1} \cap \wp^{rk_0+e-1}]_{rk_0+e} & (\hbox{by Theorem \ref{thm:inverse-system})} \\
& = & \dim \cL_2 (rk_0+e; (k_0+1)^r,rk_0+e-1) \\
& = & \dim \cL_2 (rk_0+e-r;k_0^r,rk_0+e-1-r) & (\hbox{by Remark \ref{bezout})} \\
& = & \cdots & (\hbox{by Remark \ref{bezout})}\\
& = & \dim \cL_2 (e;1^r,e-1) & (\hbox{by Remark \ref{bezout})} \\
& = & {e+2\choose 2}-{e\choose 2}-r \\
& = & 2e-r+1
\end{array}
\]
as expected.

\vskip 2mm
\underline{Case 2:} $k_0=0$. In this case we have $k=e$. It immediately follows from the equalities
$$
\dim[R/(I,L^2)]_k=\dim \cL_2 (k;1^r,k-1)= \max \left \{ 0,{k+2\choose 2}-{k\choose 2}-r \right \}
$$
and
$$
\dim[R/I]_k - \dim[R/I]_{k-2} = {k+2\choose 2} - r - {k\choose 2}.
$$
\end{proof}

Experimentally it seems that an even more general result is true, namely to remove the assumption of uniform powers, but we have not been able to prove it apart from the case of almost complete intersections mentioned earlier:

\begin{conj}
{\em For any artinian quotient of $k[x,y,z]$ generated by powers of general linear forms, and for a general linear form $L$, multiplication by $L^2$ has maximal rank in all degrees.}
\end{conj}


\section{Multiplication by $L^j$ for $3 \leq j \leq 5$}  \label{section higher j}

Let $L \in k[x,y,z]$ be a general linear form. As noted in the introduction, the main result of \cite{SS} shows that for any ideal generated by powers of linear forms, $\times L$ has maximal rank in all degrees. As we look to multiplication by successively larger powers of $L$, we will see that the maximal rank property in all degrees quickly erodes away. For uniform powers it is already known that if the linear forms are not general then $\times L^2$ does not necessarily always have maximal rank (\cite{DIV}, \cite{CHMN}). On the other hand, for ideals generated by arbitrary powers of four general linear forms (\cite{MMN-lin forms}) and for ideals generated by uniform powers of any number of general linear forms (Theorem \ref{L^2}),  $\times L^2$ does have maximal rank .

In this section we study what happens for ideals of uniform powers of four general linear forms under multiplication by $L^3, L^4$ and $L^5$. The problem is trivial for $k \leq 2$, so we assume $k \geq 3$. In Theorem \ref{L5} we assume $k \geq 4$ because the socle degree is too small when $k=3$; maximal rank holds trivially in all degrees in this case.

In this section we prove our main results, which we separate into the following theorems. A good part of the proofs will be merged using the set-up from section \ref{preparation section}.

\begin{thm} {\em (Multiplication by $L^3$.)} \label{L3}
Let $I = (L_1^k,\dots,L_4^k)$, where $L_1,\dots,L_4$ are general linear forms and $k \geq 3$.

\begin{itemize}
\item[(i)] If $k \cong 0 \hbox{\rm \ (mod } 3)$, set $k = 3 k_0$. Then for $\delta = 4k_0$ we have  $\dim [R/I]_{\delta-3} = \dim [R/I]_\delta$, and $\times L^3$ fails by exactly one to be an isomorphism between these components. In all other degrees, $\times L^3$ has maximal rank.

\medskip

\item[(ii)] If $k \not \cong 0 \hbox{\rm \ (mod } 3)$ then $\times L^3$ has maximal rank in all degrees.

\end{itemize}
\end{thm}

\begin{thm} {\em (Multiplication by $L^4$.)} \label{L4}
Let $I = (L_1^k,\dots,L_4^k)$, where $L_1,\dots,L_4$ are general linear forms and $k \geq 3$.

\begin{itemize}
\item[(i)] If $k \cong 0 \hbox{\rm \ (mod } 3)$, then $\times L^4$ has maximal rank in all degrees.

\medskip

\item[(ii)] If $k  \cong 1 \hbox{\rm \ (mod } 3)$, set $k = 3k_0 +1$. Then $\times L^4$ fails surjectivity by 1 from degree $4k_0-2$ to degree $4k_0+2$. In all other degrees $\times L^4$ has maximal rank.

\medskip

\item[(iii)] If $k \cong 2 \hbox{\rm \ (mod } 3)$, set $k = 3k_0 +2$. Then  $\times L^4$ fails injectivity by 1 from degree $4k_0-1$ to degree $4k_0+3$.  In all other degrees, $\times L^4$ has maximal rank.

\end{itemize}
\end{thm}

\begin{thm} \label{L5}  {\em (Multiplication by $L^5$.)}
Let $I = (L_1^k,\dots,L_4^k)$, where $L_1,\dots,L_4$ are general linear forms and $k \geq 4$.

\begin{itemize}
\item[(i)] If $k \cong 0 \hbox{\rm \ (mod } 3)$, set $k = 3k_0$. Then $\dim [R/I]_{4k_0-4} = \dim [R/I]_{4k_0+1}$ and $\times L^5$ fails by 3 to be an isomorphism. In all other degrees, $\times L^5$ has maximal rank.

\medskip

\item[(ii)] If $k  \cong 1 \hbox{\rm \ (mod } 3)$, set $k = 3k_0 +1$. Then $\times L^5$ fails  injectivity by 1 from degree $4k_0-3$ to degree $4k_0+2$.
In all other degrees, $\times L^5$ has maximal rank.

\medskip

\item[(iii)] If $k \cong 2 \hbox{\rm \ (mod } 3)$, set $k = 3k_0 +2$. Then $\times L^5$ fails  surjectivity by 1 from degree $4k_0-1$ to degree $4k_0+4$. In all other degrees, $\times L^5$ has maximal rank.

\end{itemize}
\end{thm}

The arguments for all of the cases of Theorems \ref{L3}, \ref{L4} and \ref{L5} are more or less the same, using Theorem \ref{thm:inverse-system}, Lemma \ref{lem:Cremona}, Remark \ref{bezout} and Lemma \ref{res of ci}. We will carefully explain one case, and leave the rest to the reader.

\begin{rem}
The multiplication by $L^j$ is reflected in the exact sequence (\ref{std exact seq}).  We have four scenarios.

\medskip

\begin{itemize}
\item To prove that $\times L^j : [R/I]_{\delta-j} \rightarrow [R/I]_\delta$ is surjective, we  have to prove that $\dim [R/(I,L^j)]_\delta~=~0$.

\medskip

\item To prove that $\times L^j : [R/I]_{\delta-j} \rightarrow [R/I]_\delta$ fails surjectivity by $\epsilon$, we have to prove that we expect surjectivity (i.e. that $C_1 - C_2 \geq 0$ in Remark \ref{find a b}) and that $\dim [R/(I,L^j)]_\delta~=~\epsilon$.

\medskip

\item To prove that $\times L^j : [R/I]_{\delta-j} \rightarrow [R/I]_\delta$ is injective, we  have to prove that
$\dim [R/(I,L^j)]_\delta = -(C_1-C_2)$.

\medskip

\item To prove that $\times L^j : [R/I]_{\delta-j} \rightarrow [R/I]_\delta$ fails injectivity by $\epsilon$, we have to prove that we expect injectivity (i.e. that $C_1 - C_2 \leq 0$ in Remark \ref{find a b}) and that $\dim [R/(I,L^j)]_\delta = -(C_1 - C_2) + \epsilon$.
\end{itemize}

\medskip

\noindent The issue of the value of $C_1-C_2$ was discussed at the end of section \ref{preparation section}. The mere fact that $\dim [R/(I,L^j)]_\delta > 0$ does not tell us which case we are in.

\medskip

\end{rem}

With minor differences, the proofs of all parts of these theorems follow the same lines. We will prove Theorem \ref{L3} (i) and Theorem \ref{L5} (i) here. ``Low" values of $k$ have to be dealt with separately, in keeping with Remark \ref{bad delta}.

Assume that $k = 3k_0$ and that $j$ is odd. Then by inspection for $k = 3,6,9$ and by Remark \ref{values of a,b} for $k \geq 12$, we have
\[
a = b =  4k_0 + \frac{j-1}{2}-1  = 4k_0 + \frac{j-3}{2}
\]
and when $\delta = a$ we get
\[
 C_1-C_2 = 0
 \]
 for both $j=3$ and $j=5$. This proves the equality of the dimensions, asserted in Theorem \ref{L3} (i) and Theorem \ref{L5} (i). We compute
\[
\begin{array}{rcll}
\dim \hbox{ coker} (\times L^j)_{4k_0+\frac{j-3}{2}} & = & \dim [R/(L_1^{3k_0}, \dots, L_4^{3k_0}, L^j)]_{4k_0+\frac{j-3}{2}} & (\hbox{by  (\ref{std exact seq}))}\\
& = & \dim [\wp_1^{k_0+\frac{j-1}{2}} \cap \dots \cap \wp_4^{k_0+\frac{j-1}{2}} \cap \wp^{4k_0- \frac{j+1}{2}}]_{4k_0+\frac{j-3}{2}} & (\hbox{by Theorem \ref{thm:inverse-system})} \\
& = & \dim [\wp_1^{k_0+\frac{j-1}{2}-1} \cap \dots \cap \wp_4^{k_0+\frac{j-1}{2}-1} \cap \wp^{4k_0-\frac{j+3}{2}}]_{4k_0+\frac{j-7}{2}} & (\hbox{by Remark \ref{bezout})}
\end{array}
\]
 If $j=3$, this is
\[
\begin{array}{rcll}
 \dim [\wp_1^{k_0} \cap \dots \cap \wp_4^{k_0} \cap \wp^{4k_0-3}]_{4k_0 -2} & = & \dim [\wp_1^{k_0-1} \cap \dots \cap \wp_4^{k_0-1} \cap \wp^{4k_0-4}]_{4k_0 -4} & (\hbox{by Remark \ref{bezout})} .
\end{array}
\]
If we denote by $\ell_i$ the equation of the line joining the point corresponding to $\wp_i$ to the point corresponding to $\wp$, then $(\ell_1 \ell_2 \ell_3 \ell_4)^{k_0-1}$ defines the unique non-zero element (up to scalar multiplication) in this vector space, so this dimension is 1 as claimed, proving the failure of isomorphism claimed in Theorem \ref{L3} (i).

Now assume $j=5$. Since $k\geq 4$ (else the socle degree is too small for $\times L^5$ to be non-trivial), we have $k_0 \geq 2$ and
\[
\begin{array}{rcll}
\dim [\wp_1^{k_0+1} \cap \dots \cap \wp_4^{k_0+1} \cap \wp^{4k_0-4}]_{4k_0 -1} & = & \dim [\wp_1^{k_0} \cap \dots \cap \wp_4^{k_0} \cap \wp^{4k_0-5}]_{4k_0 -3} & (\hbox{by Remark \ref{bezout})} \\
& = & \dim [\wp_1^{k_0-1} \cap \dots \cap \wp_4^{k_0-1} \cap \wp^{4k_0-6}]_{4k_0 -5} & (\hbox{by Remark \ref{bezout})}
\end{array}
\]
If $k_0=2$ this is easily computed to be $3$, as claimed. If $k_0 \geq 3$, we use the other part of Remark \ref{bezout} to obtain
\[
\dim [\wp_1^{k_0-2} \cap \dots \cap \wp_4^{k_0-2} \cap \wp^{4k_0-10}]_{4k_0-9}.
\]
We can continue to apply this remark until we obtain
\[
\dim [\wp_1 \cap \dots \wp_4 \cap \wp^2]_3 = 3
\]
as desired. This proves the failure of isomorphism claimed in Theorem \ref{L5} (i).

To finish Theorem \ref{L3} (i) and Theorem \ref{L5}(i), we have to prove surjectivity when $\delta \geq b+1$ and injectivity for $\delta \leq a-1$. The proof of Corollary \ref{two degrees} shows that it is enough to prove surjectivity for $\delta = b+1$ and injectivity for $\delta = a-1$.

Note that if replace $b$ by $b+1$ in the calculations above, then the same argument will yield
 \[
 \dim [\wp_1^{k_0+\frac{j-1}{2}} \cap \dots \cap \wp_4^{k_0+\frac{j-1}{2}} \cap \wp^{4k_0-\frac{j+3}{2}+1}]_{4k_0+\frac{j-7}{2}+1}
 \]
 which is
 \[
 \left \{
 \begin{array}{ll}
 \dim [\wp_1^{k_0+1} \cap \dots \cap \wp_4^{k_0+1} \cap \wp^{4k_0 -2}]_{4k_0-1} & \hbox{if } j=3 \\
 \dim [\wp_1^{k_0+2} \cap \dots \cap \wp_4^{k_0+2} \cap \wp^{4k_0 -3}]_{4k_0} & \hbox{if } j=5
 \end{array}
 \right.
 \]
 which reduces to
  \[
 \left \{
 \begin{array}{ll}
 \dim [\wp_1^{k_0} \cap \dots \cap \wp_4^{k_0} \cap \wp^{4k_0 -3}]_{4k_0-3} & \hbox{if } j=3 \\
 \dim [\wp_1^{k_0-1} \cap \dots \cap \wp_4^{k_0-1} \cap \wp^{4k_0 -6}]_{4k_0-6} & \hbox{if } j=5
 \end{array}
 \right.
 \]
 and these are both clearly 0, since for a curve of degree $d$ to have a singularity of degree $d$ at a point $p$, it must be a union of lines through $p$ (up to multiplicity), and in both cases such a union of lines cannot account for the remaining singularities. This gives our surjectivity.

 Now let $\delta = a-1$. We want to show
 \[
 \dim [R/I]_{a-1} -\dim [R/I]_{a-1-j} = \dim [R/(I,L^j)]_{a-1}.
 \]
 For $j=3$, we just have to show the result for $k=3$ and $k=6$ separately since we can use Remark \ref{values of a,b} (adjusted so that $a-1 \geq j+k-1$) for larger $k$. For $j=5$, similarly we only have to show the cases $k = 3,6,9$ separately. These are tedious but easy calculations using the above methods (or on a computer). So we can also compute the case $k_0=3$ when $j=3$, and now assume $k_0 \geq 4$ for both values of $j$.

So instead assume that  $a-1 = 4k_0 + \frac{j-1}{2} -2$ and let $C_1$ and $C_2$ be the values obtained from setting $\delta = a-1$. We get from (\ref{critical values}) that
 \[
 C_1 - C_2 = 3j \left (4k_0 + \frac{j-1}{2} -2 \right ) - 4 (3k_0)j - 3 \binom{j-1}{2} +3 = -3j.
 \]
 Now we compute
 \[
 \begin{array}{lllll}
 \dim [R/(I,L^j)]_{a-1} & =  & \dim [\wp_1^{a-k} \cap \dots \cap \wp_4^{a-k} \cap \wp^{a-j}]_{a-1} \\
 & = & \dim[\wp_1^{k_0 + \frac{j-1}{2}-1} \cap \dots \cap \wp_4^{k_0+\frac{j-1}{2}-1} \cap \wp^{4k_0 - \frac{j+1}{2}-1}]_{4k_0 + \frac{j-1}{2}-2}.
 \end{array}
 \]
 As long as $k_0 > \frac{j-1}{2}+1$,  which is true in our case, we can split off four lines.

First assume $j=3$. We have
 \[
 \dim [\wp_1^{k_0} \cap \dots \cap \wp_4^{k_0} \cap \wp^{4k_0-3}]_{4k_0-1}.
 \]
 We continue to split off four lines at a time, $n$ times. We can do this as long as
 \[
 (k_0 -n) + (4k_0 -3-4n) > 4k_0 - 1-4n
 \]
 i.e. we can do this $k_0-2$ times. We reduce to
 \[
 \dim [\wp_1^2 \cap \dots \cap \wp_4^2 \cap \wp^5]_7.
 \]
Applying Lemma \ref{lem:Cremona} twice, we see this is equal to $9 = 3j$, as desired.

When $j=5$, we have
 \[
 \dim [\wp_1^{k_0+1} \cap \dots \cap \wp_4^{k_0+1} \cap \wp^{4k_0-4}]_{4k_0}.
 \]
We can split off four lines at a time, arriving at
 \[
 \dim [\wp_1^4 \cap \dots \cap \wp_4^4 \cap \wp^8]_{12}.
 \]
Again applying Lemma \ref{lem:Cremona} twice, we obtain $15 = 3j$ as desired.

\begin{rem}\label{geq 6}
For $\times L^j$ with $j \geq 6$, we have checked on \cite{cocoa} that failure of maximal rank occurs in more than one degree for all sufficiently large values of $k$. Of course this can be confirmed with our methods.
\end{rem}

\section{Final Remarks}

As a nice application of our approach to analyze whether ideals generated by powers of linear forms have SLP we have   the following result.

\begin{pro} \label{socledegree}
Let $L_1,\cdots ,L_5\in R$ be general linear forms and $k\ge 3$. Consider the ideals $I=(L_1^k,\cdots ,L_4^k)$ and $J=(L_1^k,\cdots ,L_4^k,L_5^k)$. Then, $R/I$ and $R/J$ have the same socle degree, namely, it is $2k-2$.
\end{pro}

\begin{proof} As we pointed out at the beginning of section 3 the socle degree of $R/I$ is $2k-2$. To prove that $R/J$ has socle degree $2k-2$ it is enough to check that
$$\times L_5^k:[R/I]_{k-2}\longrightarrow [R/I]_{2k-2}$$
is not surjective or, equivalently, $[R/J]_{2k-2}\ne 0$. Arguing as in the previous sections, we have $$\dim [R/J]_{2k-2}= \dim [\wp_1^{k-1} \cap \dots \cap \wp_5^{k-1}]_{2k-2}=1$$
which proves what we want.
\end{proof}

\begin{rem}
It is natural to ask what happens for ideals generated by uniform powers of more than four general linear forms. Here is what happens for 6 general linear forms, at least experimentally using \cite{cocoa}. Here, $\times L^j$
fails to have maximal rank in all degrees for the following values of $k$ (taking $3 \leq j \leq 10$, $3 \leq k \leq 30$).

\[
\begin{array}{c|l}
j & k \\ \hline
3 & 5, 10, 15, 20, 25, 30 \\
4 & 7, 8, 12, 13, 17, 18, 22, 23, 27, 28 \\
5 & 9, 10, 11, 14, 15, 16, 19, 20, 21, 24, 25, 26, 29, 30 \\
6 & 9, 11, 12, 13, 14, 16, 17, 18, 19, 21, 22, 23, 24, 26, 27, 28, 29 \\
7 & 11, ... , 30 \\
8 & 10, 13, ... , 30 \\
9 & 12, 13, 15, ..., 30 \\
10 & 14, ..., 30

\end{array}
\]
It would be very interesting to extend the approach of Theorem \ref{L^2} to handle more than four general linear forms and prove an asymptotic result following the patterns visible here.
\end{rem}


\end{document}